\documentclass[11pt,a4paper,reqno]{amsart}
\usepackage{stmaryrd}
\usepackage[english]{babel}
\usepackage{pst-grad} 
\usepackage{pst-plot} 
\usepackage{pstricks}
\usepackage{amsmath,amssymb,mathrsfs,amsthm, mathtools}
\usepackage{tikz} 
\usepackage{mathcomp,wasysym}  
\usepackage{graphicx}  
\usepackage[all]{xy} \xyoption{arc} \xyoption{color}
\usepackage{epsfig}
\usepackage{cite}
\usepackage[a4paper,
left=2.5cm, right=2.5cm,
top=3cm, bottom=3cm]{geometry}

\normalsize
\normalsize
\usepackage[bookmarks,colorlinks]{hyperref}

\newtheorem{satz}{Proposition}[section]
 
\newtheorem{theorem}{Theorem}
\newtheorem{cor}[satz]{Corollary}

\title{On a Camassa-Holm type equation describing the dynamics of viscous fluid conduits}

\author[R. Granero-Belinch\'{o}n]{Rafael Granero-Belinch\'{o}n$^*$}
\email{rafael.granero@unican.es}
\address{Departamento  de  Matem\'aticas,  Estad\'istica  y  Computaci\'on,  Universidad  de Cantabria.  Avda.  Los  Castros  s/n,  Santander,  Spain.}

\begin{document}
\begin{abstract}
In this note we derive a new nonlocal and nonlinear dispersive equations capturing the main dynamics of a circular interface separating a light, viscous fluid rising buoyantly through a heavy, more viscous, miscible fluid at small Reynolds numbers. This equation that we termed the $g-$model shares some common structure with the Camassa-Holm equation but has additional nonlocal effects. For this new PDE we study the well-posedness together with the existence of periodic traveling waves. Furthermore, we also show some numerical simulations suggesting the finite time singularity formation.
\end{abstract}

\subjclass[2010]{35A09,35C07,35Q35}
\keywords{Well-posedness, Traveling waves, Camassa-Holm equatin, Dispersive equation}


\maketitle

\allowdisplaybreaks
\section{Introduction}
In this paper we study the following $g-$model equation 
\begin{equation}\label{model2}
g_t+2\mathcal{N}g=-\mathcal{N}(g^2)+2\mathcal{Q}(g_{zz} \mathcal{N}g- g\mathcal{N}g_{zz})  \text{ for }(z,t)\in\mathbb{R}\times[0,T]
\end{equation}
with 
$$
\widehat{\mathcal{Q}F}=\frac{1}{1+\xi^2}\hat{F}\text{ and }\widehat{\mathcal{N}F}=\widehat{\partial_z\mathcal{Q}F}=\frac{i\xi}{1+\xi^2}\hat{F}
$$
defined as multipliers in Fourier space. This equation appears as an asymptotic model of the so-called \emph{conduit equation} 
\begin{equation}\label{eq:conduit}
u_t+(u^2)_z-(u^2(u^{-1}u_t)_z)_z=0 \text{ for }(z,t)\in\mathbb{R}\times[0,T].
\end{equation}
The conduit equation models the evolution of a circular interface separating a light, viscous fluid rising buoyantly through a heavy, more viscous, miscible fluid at small Reynolds numbers. In this setting $u$ denotes the cross-sectional area of the region occupied by the lighter fluid ascending. The conduit equation \eqref{eq:conduit} has also been studied in connection with Geophysics where appears as a particular case of the so-called \emph{magma equation}
\begin{equation}\label{eq:magma}
u_t+(u^n)_z-(u^n(u^{-m}u_t)_z)_z=0.
\end{equation}
This equation describes the vertical motion of molten rock up a viscously deformable dyke in the earth's crust. Both equations have been subject of great efforts and a large number of works. In particular, we refer to the following non-exhaustive list \cite{ambrose2018existence,barcilon1989solitary,gavrilyuk2024conduit,
lowman2013dispersive,
maiden2016modulations
,mao2023experimental,olson1986solitary
,scott1984magma,scott1986magma,whitehead1986korteweg} and the (many other) references therein. Equation \eqref{model2} has a similar flavour to the well-known Camassa-Holm equation. In fact, if we replace $\mathcal{Q}$ by the identity, then we recover the CH equation (up to certain constants).

The purpose of this note is three-fold. First, we derive the $g-$model. This derivation is in the same spirit as the work by Takahashi, Sachs and Satsuma \cite{takahashi1990properties} where they find the KdV as an asymptotic model in a different regime. Then we study the well-posedness and finite time singularity formation. In particular we prove the following result:
\begin{theorem}\label{theorem1}
Let $g_0\in H^2(\mathbb{R})$ be a zero mean initial data. Then there exists a solution of \eqref{model2} satisfying
$$
g\in C([0,T],H^2(\mathbb{R}))
$$
for certain $0<T.$
\end{theorem}
We observe that the same result holds in the case where the domain is $[-\pi,\pi]$ with periodic boundary conditions.

The question of the existence of traveling waves for the full magma equation has been an object of study at least since the 90's \cite{takahashi1990properties}.

Finally, we prove the existence of periodic traveling waves:
\begin{theorem}\label{theorem2}
For any $1\leq m\in\mathbb{Z}$ there exists a one dimensional curve $s\mapsto (c_s, \varphi_s)$, with $s\in I$, such that
$$
\varphi_s(x)\in X_m:=\left\{f\in C^{2,\alpha}([0,2\pi]),\quad f(\xi)=\sum_{k\geq 1}f_k\cos(mk\xi)\text{ with norm }\|f\|_{X_m}=\|f\|_{C^{2,\alpha}}\right\},
$$
is a $m-$fold traveling wave solution to \eqref{model2} with constant speed $c_s$.
\end{theorem}
In particular, although the previous theorem ensures that there exist infinitely many traveling wave solutions for the $g-$model, the proof is not constructive. Thus, the specific shape of the traveling waves remains non-explicit.

\section{Derivation of the g-model for the conduit}
We assume that 
$$
u(z,t)=1+\varepsilon h(z,t)
$$
and $0<\varepsilon\ll 1$. In this case, \eqref{eq:conduit} reads
\begin{equation}\label{eq:full}
h_t-h_{tzz}+2h_z=-2\varepsilon hh_z-\varepsilon h_{zz} h_t+\varepsilon hh_{tzz}.
\end{equation}
We make the ansatz
\begin{equation*}
h(z,t)=\sum_{\ell=0}^\infty \varepsilon^\ell h^{(\ell)}(z,t).
\end{equation*}
As a consequence, we obtain the following cascade of linear partial differential equations:
\begin{equation*}
h^{(\ell)}_t-h^{(\ell)}_{tzz}+2h^{(\ell)}_z=-2\sum_{j=0}^{\ell-1} h^{(j)}h_z^{(\ell-j-1)}-\sum_{j=0}^{\ell-1}h^{(j)}_{zz} h_t^{(\ell-j-1)}+ \sum_{j=0}^{\ell-1}h^{(j)}h_{tzz}^{(\ell-j-1)}.
\end{equation*}
This cascade of linear problems can be solved for recursively. In particular, we have that the first two terms in this cascade verify
$$
h^{(0)}_t-h^{(0)}_{tzz}+2h^{(0)}_z=0,
h^{(1)}_t-h^{(1)}_{tzz}+2h^{(1)}_z=-2 h^{(0)}h_z^{(0)}-h^{(0)}_{zz} h_t^{(0)}+ h^{(0)}h_{tzz}^{(0)}.
$$
We observe that we can equivalently write
$$
h^{(0)}_t=-2\mathcal{N}h^{(0)}.
$$
As a consequence, we can also obtain
$$
h^{(1)}_t-h^{(1)}_{tzz}+2h^{(1)}_z=-2 h^{(0)}h_z^{(0)}+2h^{(0)}_{zz} \mathcal{N}h^{(0)}-2 h^{(0)}\mathcal{N}h_{zz}^{(0)}.
$$
If we define now
$$
g(z,t)=h^{(0)}(z,t)+\varepsilon h^{(1)}(z,t),
$$
we conclude that, up to $O(\varepsilon^2)$ corrections, $g$ solves 
\begin{equation}\label{model}
g_t-g_{tzz}+2g_z=-2 gg_z+2g_{zz} \mathcal{N}g-2 g\mathcal{N}g_{zz}.
\end{equation}
We observe that this model can also be written as \eqref{model2}.
We call this model the $g-$model of the conduit equation.

In our opinion, this equation deserves attention by itself as its mathematical properties are unknown and the behaviour is far from trivial as the numerics below evidence. Furthermore, we think that the mathematical properties of this system can serve to further understand the full conduit and magma equation in the non-vacuum regime. In particular, the existence of traveling wave solutions is a shared property between the $g-$ model and the full conduit equation. Furthermore, we think that the actual shape of such traveling waves should be very similar. However, whether the $g-$ equation is a so-called \emph{fully justified} asymptotic model is still an open question that requires its own study in the spirit of the works that do so for models of water waves.

\section{Well-posedness and numerical evidence of finite time singularities}
This $g-$model is locally well-posed in Sobolev spaces.
\begin{proof}[Proof of Theorem \ref{theorem1}]
One could prove this result using the ideas as in \cite{ambrose2018existence}. However, it can be noted that, using the regularizing properties of $\mathcal{N}$ and $\mathcal{Q}$, the Picard theorem in Banach spaces can be directly applied leading to the local existence of classical solution.
\end{proof}

Furthermore, the same result holds in the periodic case. In fact, one can similarly prove

\begin{cor}\label{theorem1b}
Let $g_0\in H^2(\mathbb{T})$ be a zero mean initial data. Then there exists a solution of \eqref{model} satisfying
$$
g\in C([0,T],H^2(\mathbb{T}))
$$
for certain $0<T.$
\end{cor}
and, using again Picard theorem now in H\"older spaces,`
\begin{cor}\label{theorem1c}
Let $0<\alpha<1$ and $g_0\in C^{2,\alpha}(\mathbb{R})$ be a zero mean initial data. Then there exists a solution of \eqref{model} satisfying
$$
g\in C([0,T],C^{2,\alpha}(\mathbb{R}))
$$
for certain $0<T.$
\end{cor}

Using a Fourier collocation spectral method to discretize the spatial variable together with the standard \emph{ode45} Runge-Kutta time integrator we can approximate the space periodic analog of \eqref{model2}. In particular, for $N=2^{13}$ spatial notes and initial data $g_0(x)=\sin(x)\text{ and }g_{0}(x)=\sin(x)\cos(x),$
we find the following
\begin{figure}[h!]
\includegraphics[scale=0.15]{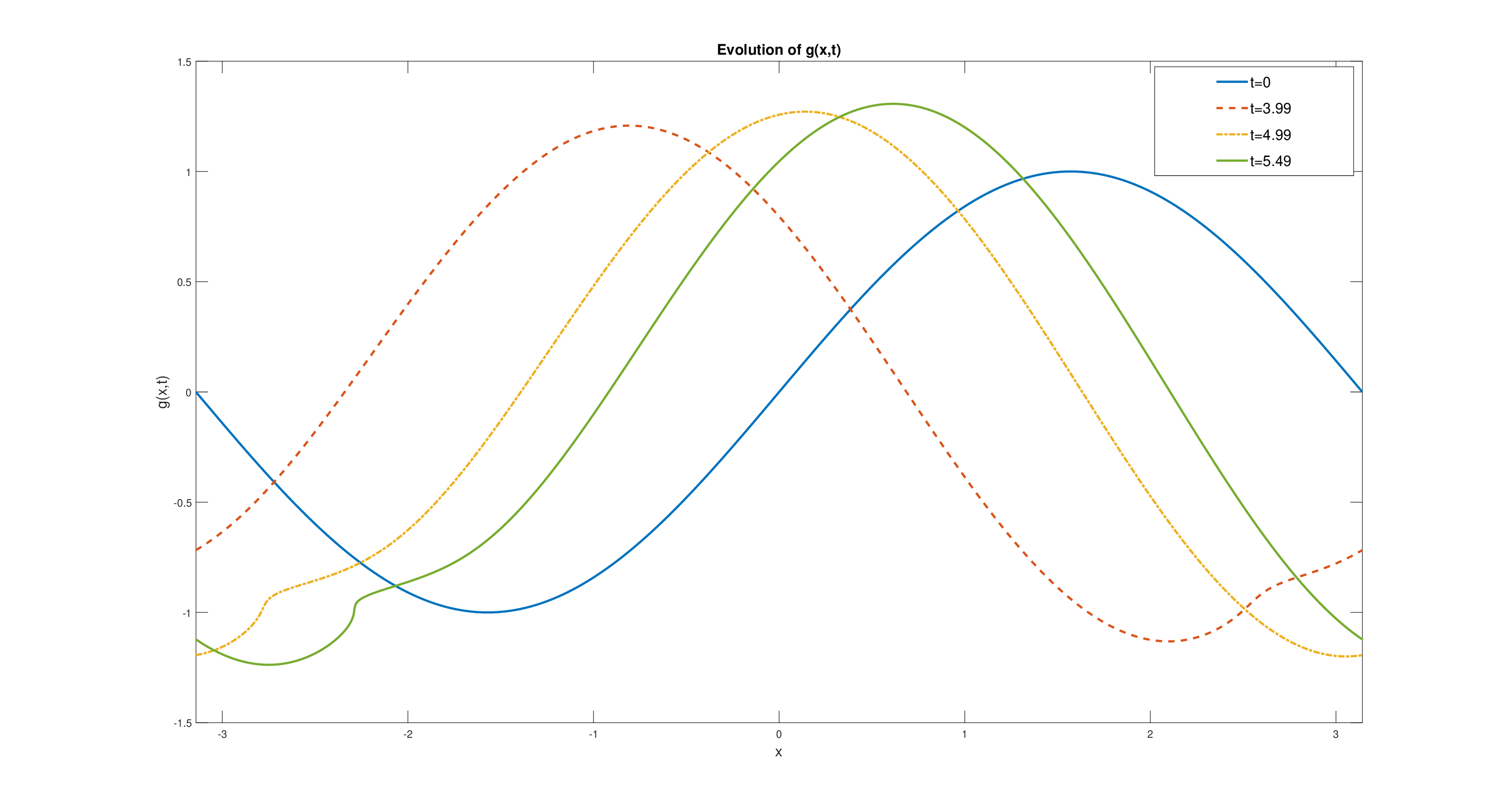}
\includegraphics[scale=0.15]{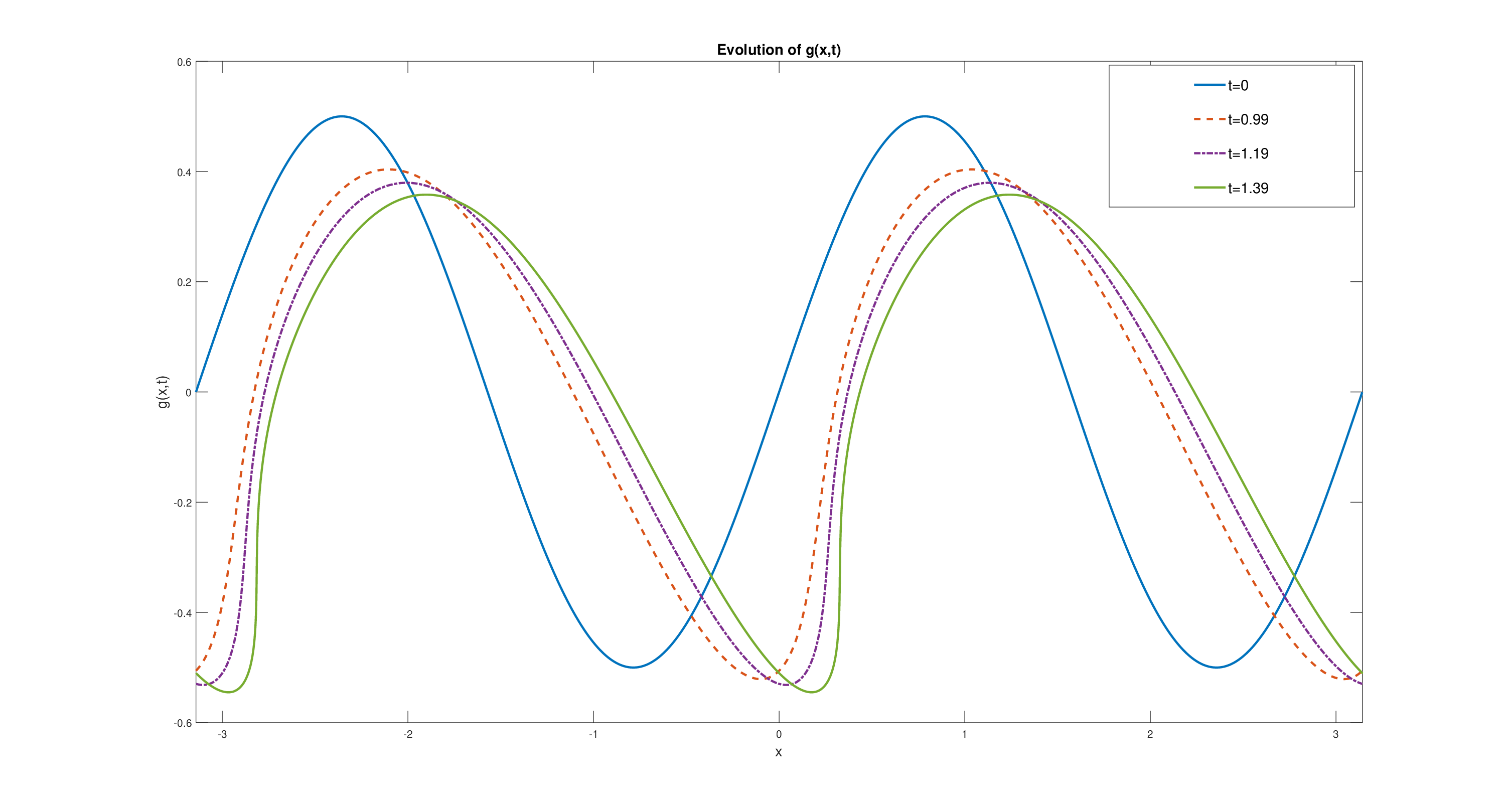}
\end{figure}

Let us remark that it is however unclear whether the solution blows up in finite time and that the previous numerical simulations are not conclusive. Furthermore, as noted in \cite{lowman2013dispersive}

\section{Existence of periodic traveling waves}
In this section we prove Theorem \ref{theorem2}. We start with certain concepts and statements required in bifurcation theory. Let $X$ and $Y$ be two Banach spaces. A continuous linear mapping $T:X\rightarrow Y,$  is a  Fredholm operator if it fulfills the following properties,
\begin{enumerate}
\item $\textnormal{dim Ker}\,  T<\infty$,
\item $\textnormal{Im}\, T$ is closed in $Y$,
\item $\textnormal{codim Im}\,  T<\infty$.
\end{enumerate}
The integer $\textnormal{dim Ker}\, T-\textnormal{codim Im}\, T$ is called the Fredholm index of $T$. We also remark that the index of a Fredholm operator remains unchanged under compact perturbations.

We now state the Crandall-Rabinowitz theorem \cite{crandall1971bifurcation}:
\begin{theorem}[Crandall-Rabinowitz Theorem]\label{CR}
    Let $X, Y$ be two Banach spaces, $V$ be a neighborhood of $0$ in $X$ and $F:\mathbb{R}\times V\rightarrow Y$ be a function with the properties,
    \begin{enumerate}
        \item $F(\lambda,0)=0$ for all $\lambda\in\mathbb{R}$.
        \item The partial derivatives  $\partial_\lambda F$, $\partial_fF$ and  $\partial_{\lambda}\partial_fF$ exist and are continuous.
        \item The operator $\partial_f F(\lambda_{0},0)$ is Fredholm of zero index and $\textnormal{Ker}(\partial_f F(\lambda_{0},0))=\langle f_0\rangle$ is one-dimensional. 
                \item  Transversality assumption: $\partial_{\lambda}\partial_fF(\lambda_{0},0)f_0 \notin \textnormal{Im}(\partial_fF(\lambda_{0},0))$.
    \end{enumerate}
    If $Z$ is any complement of  $\textnormal{Ker}(\partial_fF(\lambda_{0},0))$ in $X$, then there is a neighborhood  $U$ of $(\lambda_{0},0)$ in $\mathbb{R}\times X$, an interval  $(-a,a)$, and two continuous functions $\Phi:(-a,a)\rightarrow\mathbb{R}$, $\beta:(-a,a)\rightarrow Z$ such that $\Phi(0)=\lambda_{0}$ and $\beta(0)=0$ and
    $$F^{-1}(0)\cap U=\{(\Phi(s), s f_0+s\beta(s)) : |s|<a\}\cup\{(\lambda,0): (\lambda,0)\in U\}.$$
\end{theorem}

The bifurcation approach to prove the existence of traveling waves has been used by different authors (see \cite{alonso2024existence,alonso2024traveling} and the references therein).

After inserting the traveling wave ansatz for a fixed speed $c\in\mathbb{R}$
$$
g(x,t)=\varphi(x-ct)
$$
into \eqref{model}, we obtain
\begin{align}
F[c,\varphi](\xi)&=-c\varphi'+2\mathcal{N}\varphi+\mathcal{N}(\varphi^2)-2\mathcal{Q}(\varphi'' \mathcal{N}\varphi- \varphi\mathcal{N}\varphi'')=0, \quad \xi\in[0,2\pi].\label{definitionF}
\end{align}
Following \cite{alonso2024traveling}, we define the functional spaces
\begin{align*}
X:=\left\{h\in C^{2,\alpha}([0,2\pi]),\quad h(\xi)=\sum_{k\geq 1}h_k\cos(k\xi)\text{ with norm }\|h\|_{X}=\|h\|_{C^{2,\alpha}}\right\},\\
Y:=\left\{h\in C^{1,\alpha}([0,2\pi]),\quad h(\xi)=\sum_{k\geq 1}h_k\sin(k\xi)\text{ with norm }\|h\|_{Y}=\|h\|_{C^{1,\alpha}}\right\},\\
Z:=\left\{h\in C^{3,\alpha}([0,2\pi]),\quad h(\xi)=\sum_{k\geq 1}h_k\sin(k\xi)\text{ with norm }\|h\|_{Z}=\|h\|_{C^{3,\alpha}}\right\}. 
\end{align*}
We observe that the embedding $Z\hookrightarrow Y$ is compact. We now provide the proof of Theorem \ref{theorem2}.

\begin{proof}[Proof of Theorem \ref{theorem2}]
We have to check that the hypotheses of the Crandall-Rabinowitz theorem hold. We start by showing that the operator $F:\mathbb{R}\times X\rightarrow Y$ given in \eqref{definitionF} is well-defined and $\mathscr{C}^1(\mathbb{R}\times X\rightarrow Y)$.

Let us start checking that $F$ is well-defined. First of all, we observe that if $\varphi$ is an even function, $F[c,\varphi]$ is an odd function. Indeed, we observe that
$$
\mathcal{N}(\varphi^2)=\mathcal{Q}(2\varphi\varphi').
$$
A standard application of classical Schauder estimates for $\mathcal{Q}$ and $\mathcal{N}$ (see \cite{gilbarg1977elliptic}) shows that 
$$
F[c,\varphi]:X\mapsto Y.
$$
Furthermore, we find that
$$
\|F[c,\varphi]\|_{Y}\leq C\left(\|\varphi\|_{X}^2+\|\varphi\|_{X}\right).
$$
In order to prove that $F\in \mathscr{C}^1(\mathbb{R}\times X\rightarrow Y)$ we have to show that the Gateaux derivative of $F$ is continuous. In fact, the Gateaux derivative of $F$ reads
$$
\partial_\varphi F[c,\varphi]h=-ch'+2\mathcal{N}h+\mathcal{N}(2\varphi h)-2\mathcal{Q}(\varphi'' \mathcal{N}h- \varphi\mathcal{N}h''+h'' \mathcal{N}\varphi- h\mathcal{N}\varphi''),
$$
and verifies
\begin{equation}\label{FC1}
\|\partial_\varphi F[c,\varphi_1]h-\partial_\varphi F[c,\varphi_2]h\|_{Y}\leq C\|h\|_{X}\|\varphi_1-\varphi_2\|_{X}. 
\end{equation}

We observe that the linear operator can be written as
$$
L=L^1+L^2\text{ with }L^1=-ch', L^2=2\mathcal{N}h.
$$
Using again Schauder estimates we find that
$$
L^2:X\rightarrow Z\subset\subset Y\text{ and }L^1:X\rightarrow Y \text{ is an isomorphism. }
$$
Thus, $\partial_\varphi F[c,0]$ is a Fredholm operator of zero index. We now characterize the kernel of the linear operator. For $h\in X$ we find that
$$
\partial_\varphi F[c,0]h=\sum_{j=1}^\infty h_j\sin(jx)\left(2\frac{j}{1+j^2}-cj\right).
$$
Thus, for 
$$
c_n=\frac{2}{1+n^2}
$$
we have that
$$
Ker(\partial_\varphi F[c,0])=\text{span}(\cos(jx))
$$
and, recalling that the linearized operator is Fredholm operator of zero index,
$$
Y/Img(\partial_\varphi F[c,0])=\text{span}(\sin(jx)).
$$
The transversality condition is then satisfied because
$$
\partial_c \partial_\varphi F[c_n,0]h=\sum_{j=1}^\infty h_j\sin(jx)\left(-j\right)
$$
and, if $h\in Ker(\partial_\varphi F[c,0])$, then
$$
\partial_c \partial_\varphi F[c_n,0]h=-nh_n\sin(nx).
$$

We have checked that the Crandall-Rabinowitz theorem can be applied in our equation \eqref{model2}. Fix $m\geq 1$. In order to prove Theorem \ref{theorem2}, we define
\begin{align*}
X_m:=\left\{f\in C^{2,\alpha}([0,2\pi]),\quad f(\xi)=\sum_{k\geq 1}f_k\cos(mk\xi)\text{ with norm }\|f\|_{X_m}=\|f\|_{C^{2,\alpha}}\right\},\\
Y_m:=\left\{f\in C^{1,\alpha}([0,2\pi]),\quad f(\xi)=\sum_{k\geq 1}f_k\sin(mk\xi)\text{ with norm }\|f\|_{Y_m}=\|f\|_{C^{1,\alpha}}\right\}.
\end{align*}
Due to the previous argument (see \cite{alonso2024traveling} for more details), Crandall-Rabinowitz theorem can be applied obtaining Theorem \ref{theorem2}.
\end{proof}

\section{Conclusion}
The derivation and study of right asymptotic models for the magma and conduit equation dates back to the 80's with the works of Whitehead and Helfrich \cite{whitehead1986korteweg} In this paper we have derived a new (and different) asymptotic model called the $g-$model for the conduit and magma equations. We have also establishd that such model is well-posed in a variety of functional spaces.

Furthermore, as noted in Lowman and Hoefer \cite{lowman2013dispersive}, even if in certain regime the dynamics can be accurately captured by the KdV equation, the full model and the real life experiments seem to present dispersive shocks. In this regards, we hope that our model can also describe such behaviour. This remains one interesting open question yet to be studied.

The existence and non-existence of solitary and traveling waves for the conduit and magma equations have been studied, at least, since the 80's with the works of Olson and Christensen \cite{olson1986solitary}, Barcilon and Lovera \cite{barcilon1989solitary} or Takahashi, Sachs and Satsuma \cite{takahashi1990properties}. In this regards, we show that our $g-$model also captures this traveling wave behavior.

Besides the interest in the conduit equation, it is expected that the ideas that we have developed here could pottentiall be applied to other systems such as equation (19) in \cite{carpio2020incorporating}

\section*{Acknowledgments}
The author was supported by the project "Mathematical Analysis of Fluids and Applications" Grant PID2019-109348GA-I00 and "An\'alisis Matem\'atico Aplicado
y Ecuaciones Diferenciales" Grant PID2022-141187NB-I00 funded by MCIN/AEI/10.13039/501100011033/FEDER, UE. 

This publication is part of the project PID2022-141187NB-
I00 funded by MCIN/AEI/10.13039/501100011033.


\end{document}